\documentclass[twocolumn, amsthm]{autart}
\pdfoutput=1

\usepackage[utf8]{inputenc}

\usepackage{mathtools}
  \mathtoolsset{showonlyrefs,centercolon}
\usepackage{amssymb}
\usepackage{dsfont}
\usepackage{nicefrac}
\usepackage[np,autolanguage]{numprint}
\usepackage{IEEEtrantools}

\usepackage[rgb, usenames, svgnames]{xcolor}

\usepackage{enumerate}
\usepackage{paralist}
\usepackage[multidot]{grffile}
\usepackage{needspace}
\usepackage{xifthen}

\usepackage{natbib}

\usepackage[colorlinks=true, linkcolor=black, citecolor=black]{hyperref}

\newcommand{\trans}{^\intercal} 
\newcommand{\inv}{^{-1}} 

\newcommand{\downto}{\downarrow} 

\newcommand{\ident}{\ensuremath{\mathrm{I}}} 
\newcommand{\dd}{\ensuremath{\mathrm{d}}} 
\newcommand{\reals}{\ensuremath{\mathds{R}}} 
\newcommand{\realspos}{\ensuremath{\reals_{>0}}} 
\newcommand{\realsnonneg}{\ensuremath{\reals_{\geq 0}}} 
\newcommand{\ereals}{\ensuremath{\overline{\reals}}} 
\newcommand{\naturals}{\ensuremath{\mathds{N}}} 
\DeclareMathOperator*{\closure}{clos} 
\DeclareMathOperator{\tr}{tr} 
\DeclareMathOperator{\diverg}{div} 

\newcommand{\paren}[1]{\ensuremath{\left(#1\right)}}

\newcommand{\fparen}[1]{\ensuremath{\!\left(#1\right)}}

\newcommand{\brackt}[1]{\ensuremath{\left[#1\right]}}

\newcommand{\bracet}[1]{\ensuremath{\left\{#1\right\}}}

\newcommand{\setb}[2]{\ensuremath{\bracet{#1 \colon #2}}}

\newcommand{\norm}[2][]{\ensuremath{\left\lVert#2\right\rVert}_{#1}}

\newcommand{\normnm}[2][]{\ensuremath{\lVert#2\rVert}_{#1}}

\newcommand{\innerp}[2][]{\ensuremath{\left\langle#2\right\rangle}_{#1}}

\newcommand{\abs}[1]{\ensuremath{\left\lvert#1\right\rvert}}

\newcommand{\Prob}[2][]{%
    \ensuremath{P\ifthenelse{\isempty{#1}}{}{_{#1}\!}\fparen{#2}}%
}

\newcommand{\pr}[2][]{%
  \ensuremath{p\ifthenelse{\isempty{#1}}{\!}{_{#1}\negmedspace}\paren{#2}}%
}

\newcommand{\E}[2][]{\ensuremath{\mathrm{E}_{#1}\negmedspace\left[#2\right]}}

\newcommand{\normald}[1]{\ensuremath{\mathcal{N}\!\paren{#1}}}

\newcommand{\cond}{\ensuremath{\,\middle\vert\,}}

\newcommand{\inputtikzpicture}[1]{%
  \includegraphics{fig/#1.pdf}%
}

\newcommand{\events}{\ensuremath{\mathcal{E}}}
\newcommand{\partition}{\ensuremath{\mathcal{P}}}
\newcommand{\mesh}{\ensuremath{\bar\delta}}

\newcommand{\aset}{\ensuremath{\mathcal{A}}}
\newcommand{\yset}{\ensuremath{\mathcal{Y}}}

\newcommand{\etaspace}{\ensuremath{\mathcal{H}}}
\newcommand{\ispace}{\ensuremath{\mathcal{I}}}

\newcommand{\sspace}{\ensuremath{\mathcal{S}}}
\newcommand{\tspace}{\ensuremath{\mathcal{T}}}

\newcommand{\xspace}{\ensuremath{\mathcal{X}}}

\newcommand{\tmeas}{\ensuremath{\tspace_\mathrm{m}}}
\newcommand{\energy}{\ensuremath{\mathrm{e}}}

\newcommand{\etube}[1][\phi]{\ensuremath{\mathcal{B}^\epsilon_{#1}}}
\newcommand{\entube}[1][\eta,x_0]{\ensuremath{\mathcal{C}^\epsilon_{#1}}}

\date{}

\begin{document}

\begin{frontmatter}
  \title{Maximum \emph{a Posteriori} \ State Path Estimation: Discretization Limits and their Interpretation\thanksref{footnoteinfo}}

  \thanks[footnoteinfo]{This work has been financially supported by the Brazilian agencies CAPES, CNPq, FAPEMIG and FINEP.}
  
  \author{Dimas Abreu Dutra}\ead{ddutra@cpdee.ufmg.br},\hfill
  \author{Bruno Otávio Soares Teixeira}\ead{brunoot@ufmg.br},\hfill
  \author{Luis Antonio Aguirre}\ead{aguirre@cpdee.ufmg.br}.

  \address{Programa de Pós-Graduação em Engenharia Elétrica - Universidade Federal de Minas Gerais - Av. Antônio Carlos 6627, 31270-901, Belo Horizonte, MG, Brasil}
  
  
  \begin{abstract}
    Continuous-discrete models with dynamics described by stochastic differential equations are used in a wide variety of applications.
    For these systems, the maximum \emph{a posteriori} (MAP) state path can be defined as the curves around which lie the infinitesimal tubes with greatest posterior probability, which can be found by maximizing a merit function built upon the Onsager--Machlup functional.
    A common approach used in the engineering literature to obtain the MAP state path is to discretize the dynamics and obtain the MAP state path for the discretized system.
    In this paper, we prove that if the trapezoidal scheme is used for discretization, then the discretized MAP state path estimation converges hypographically to the continuous-discrete MAP state path estimation as the discretization gets finer.
    However, if the stochastic Euler scheme is used instead, then the discretized estimation converges to the minimum energy estimation.
    The minimum energy estimates are, in turn, proved to be the state paths associated with the MAP \emph{noise} paths, which in some cases differ from the MAP \emph{state} paths.
    Therefore, the discretized MAP state paths can have different interpretations depending on the discretization scheme used.
  \end{abstract}
\end{frontmatter}

\section{Introduction}
\label{sec:introduction}

A wide variety of phenomena of engineering interest are continuous-time in nature and can be modelled by stochastic differential equations (SDEs).
In practical situations, the measurements of these systems are usually sampled, leading to continuous-discrete models, i.e., models with continuous-time dynamics and measurements taken in discrete-time.
For this class of models, the maximum \emph{a posteriori} (MAP) state paths can be defined as the curves which have the greatest posterior sojourn probability, i.e., that maximize the probability that the state process lies inside an infinitesimal tube around them.
The asymptotic prior probability of $\epsilon$-tubes around state paths, as $\epsilon$ vanishes, is given by the Onsager--Machlup functional \citep{fujita1982omf, zeitouni1987map, capitaine1995omf}, which can be used to construct a variational problem for MAP state path estimation in continuous-discrete models \citep{aihara1999mln}.
It should be noted, however, that variational MAP state path estimation using the Onsager--Machlup functional is not common in the engineering literature.
One of the few exceptions is the work of \citet{aihara1999mem,aihara1999mln}, where only experiments with simulated data are shown.

In the context of discrete-time dynamical systems, there exist several methods for maximum \emph{a posteriori} (MAP) state path estimation \citep{bell1994iks, godsill2001map, aravkin2011llr, monin2013mte}.
The MAP state path consists of the joint posterior mode of the states over all time indices, given a series of measurements.
These estimators have been the focus of many recent publications due to their robustness properties and larger applicability than nonlinear Kalman smoothers \citep{farahmand2011drs, aravkin2011llr, aravkin2011lus, aravkin2012rtf, aravkin2012sct, dutra2012jmaps, monin2013mte}.

To use discrete-time MAP state path estimation in the context of continuous-discrete systems, the system dynamics needs to be converted to discrete-time.
However, while the discretization of linear systems can be exact, for general nonlinear SDEs it is necessary to employ approximations such as the stochastic Euler, trapezoidal or It\=o--Taylor schemes \citep{kloeden1992nss}.
Applications of MAP state path estimation on discretized continuous-discrete systems described by SDEs are presented in the work of \citet{bell2009icn} and \citet{aravkin2011llr, aravkin2012rtf}.
In these applications, the former used the exact discretization for linear systems, while the latter used the Euler discretization for nonlinear systems.

Discrete-time approximations of SDEs improve as the discretization step decreases.
Applications in Kalman filtering that use such discretization schemes, for example, usually require fine discretizations to produce meaningful results \citep{sarkka2012cdc}, especially when the measurements are sparse.
An important question then arises: how are the MAP state paths of the discretized systems related to the MAP state path of the original system?
Do the discretized estimates converge to continuous-time functions?

\begin{figure}[b]
  \centering
  \inputtikzpicture{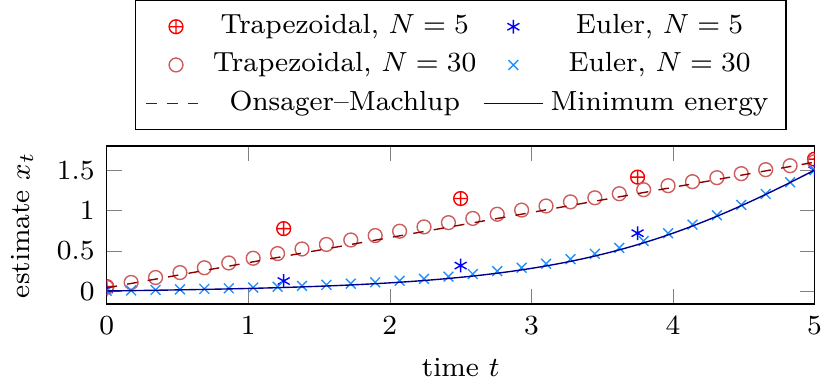}
  \caption{
    Comparison of the different MAP estimators for a simple SDE.
    The discretizations with the exact transition densities lie on the state path obtained with the Onsager--Machlup, for all $N$, and are not shown to avoid clutter.
  }
  \label{fig:benes_estimators}
\end{figure}

In this paper, we use the theory of consistent discretizations \citep{polak1997oac} to show that, under some regularity conditions, the MAP estimates using the Euler and trapezoidal schemes converge to maxima of variational problems.
The limit variational problem for the trapezoidal discretization is the MAP state path estimation problem.
For the Euler discretization, however, the limit variational problem is the minimum energy estimation problem \citep[Sec.~3.7]{cox1963esv}, whose merit function is similar to that of the MAP problem but lacks a correction term which is related to the amplification of noise by the drift.
The minimum energy estimator was, for many years, believed to be a good MAP estimator \citep[p.~234]{zeitouni1987map}.
Its estimates are MAP state paths, however, only when the correction term is
constant.
In this paper we prove that, in the general case, the minimum energy estimates correspond to the state paths associated with the MAP \emph{noise} paths.

The main contributions of this paper are%
\begin{inparaenum}[\it i)]%
\item providing the statistical interpretation for minimum energy estimates as the state path associated with the MAP noise path,
\item proving a link between variational and discretized MAP state path estimation, and
\item proving that the limits of the discretized MAP state path estimation problems depend on the choice of discretization used.
\end{inparaenum}
The remainder of this paper is organized as follows.
In Sec.~\ref{sec:motivation} a simple motivating example is presented to illustrate the main ideas of the paper.
In Sec.~\ref{sec:bayesian_decision} we present a formal definition of mode, in the framework of Bayesian decision theory, that generalizes well to infinite-dimensional spaces.
We then define the continuous-discrete estimation problems in Secs.~\ref{sec:probability_tubes}--\ref{sec:continuous-discrete_map} and the discretized problems in Sec.~\ref{seq:discretized_map}.
Finally, in Sec.~\ref{sec:limit_discretized_problems} the hypo-convergence of the problems is proved.
Simulated examples are presented in Sec.~\ref{sec:simulated_experiments} and the conclusions and future work are presented in Sec.~\ref{sec:conclusions}.

\section{Motivating example}
\label{sec:motivation}
Take the simple scalar SDE below, for which the exact transition probability is known \citep{daum1986efd}:
\begin{equation}
  \label{eq:benes_sde}
  \dd X_t = \tanh(X_t)\,\dd t + \dd W_t.
\end{equation}
Then, consider a state path estimation problem where the only information available is the initial state prior distribution and a single measurement at $t=5$ with value $y=\np{1.5}$.
The initial state and measurement are normally distributed with mean and covariance as follows:
\begin{align}
  X_0&\sim\normald{0,\, 0.16}, & 
  Y| X_5 &\sim \normald{X_5,\, 0.16}.
\end{align}

If we discretize this SDE at $N$ equally spaced time-points using the exact transition density, the Euler scheme and the trapezoidal scheme, it can be seen that the MAP state path estimates converge to continuous-time functions as $N\to\infty$.
The estimates with the trapezoidal and exact transition densities have the same limit, which is the MAP state path obtained using the Onsager--Machlup functional.
\emph{The estimates with the Euler discretization, however, converge to a different curve, which is the minimum energy state path.}
These results are summarized in Fig.~\ref{fig:benes_estimators}.
The maximization problems corresponding to each solution are presented in Secs.~\ref{sec:continuous-discrete_map}--\ref{seq:discretized_map}.

This example shows that, for some systems, the Euler discretization is not appropriate for approximating the joint density of the state path, even with very fine discretizations.
The MAP state paths obtained with the Euler scheme converge to the minimum energy estimates, not to the true MAP state paths obtained with the exact transition density or the Onsager--Machlup functional.

\section{Bayesian decision and MAP estimation}
\label{sec:bayesian_decision}
This section is a short introduction to Bayesian estimation theory to argue that the maximization of the probability of tubes of vanishing radius is a natural extension of the concept of MAP estimation to continuous-time stochastic processes.
We also formalize the concept with a clear mathematical definition that takes multimodality into account.

To begin, let $(\Omega, \events, P)$ be the probability space on which all random variables are defined.
Random variables will be denoted by uppercase letters and their values by lowercase, so that if $X\colon\Omega\to\xspace$ is a $\xspace$-valued random variable, $x\in\xspace$ will be used to refer to specific values it might take.
The same applies for stochastic processes.
The dependency on the random outcome $\omega\in\Omega$ will be omitted when unambiguous, to simplify the notation.

Bayesian estimation can be formulated as taking a decision which minimizes an expected loss, given an observation.
Say $X$ is a random variable, which is to be estimated, and $\aset$ is an observed event, summarizing all the information available.
The Bayesian choice \citep{robert2001bc} for the estimate $\hat x$ corresponds to a value which minimizes of the expected loss, i.e., 
\begin{equation}
  \label{eq:bayesian_loss}
  \textstyle
  \E{\ell(X, \hat x)\cond \aset} = \inf_x\E{\ell(X, x)\cond \aset},
\end{equation}
where $\ell$ is an $\reals$-valued function which represents the loss associated with each outcome of $X$ and choice for $x$.

Several known loss functions correspond to certain statistics of the posterior distribution.
When the loss is the absolute distance and squared distance between the estimate and the outcome of the random variable,
\begin{align}
  \ell_1(X,x)&:=\norm{X-x}, &
  \ell_2(X,x)&:=\norm{X-x}^2\!,
\end{align}
then the Bayesian estimator is the posterior median and mean, respectively.
For discrete random variables, i.e., when the image of $X$ is a countable set, the posterior modes are the Bayesian estimates associated with the 0--1 loss function, given by
\begin{equation}
  \label{eq:0_1_loss_discrete_rv}
  \ell_{01}(X,x):=
  0 \text{ if }X=x\text{, }1\text{ otherwise}.
\end{equation}
For $\reals^n$-valued random variables which admit continuous posterior densities, the expected value of $\ell_{01}$ is always unity, so we must instead consider a family of 0--1 loss functions given by
\begin{equation}
  \label{eq:0_1_loss_defn}
  \ell_\epsilon(X,x):=
  0 \text{ if }\norm{X - x} \leq \epsilon\text{, }1\text{ otherwise}.
\end{equation}
One possible definition for the MAP estimate is then the limit of Bayesian estimates $\hat x_\epsilon$ associated with the losses $\ell_\epsilon$, as $\epsilon\downto 0$ \citep[Sec.~4.1.2]{robert2001bc}.
It can be shown that the limit of any convergent subsequence of $\hat x_\epsilon$ is always a maximum of the posterior density.

However, the above definition for MAP estimate has several problems.
First, convergence of the Bayesian estimates is hard to guarantee.
Even the existence of convergent subsequences requires additional assumptions on the distribution of the random variable.
In addition, there can exist maxima of posterior density that cannot be limits of the Bayesian estimates.
These limitations are often overlooked because the \emph{de facto} definition of MAP estimates is the location of the maxima of the posterior density, i.e., the modes of the posterior distribution.
For random variables in infinite-dimensional spaces, for which there is no probability density in the usual sense, this \emph{de facto} definition cannot be used.
To overcome the problems mentioned above, we present below a formal definition of mode which agrees with the \emph{de facto} one and generalizes well to infinite-dimensional spaces.

\begin{defn}[mode]
  \label{th:mode_defn}
  Let $\xspace$ be a normed vector space and $X$ be a $\xspace$-valued random variable.
  An element $\hat x\in\xspace$ is a mode of $X$ if, for all $x\in\xspace$,
  \begin{equation}
    \label{eq:mode_defn}
    \limsup_{\epsilon\downto 0}
    \frac{\Prob{\norm{X-x}\leq \epsilon}}
    {\Prob{\norm{X-\hat x}\leq \epsilon}} \leq 1.
  \end{equation}
\end{defn}

\begin{defn}
  \label{th:map_defn}
  The maximum \emph{a posteriori} (MAP) estimates of a random variable are the modes of its posterior distribution, according to Definition \ref{th:mode_defn}.
\end{defn}

Definitions \ref{th:mode_defn} and \ref{th:map_defn} can be interpreted in the framework of Bayesian decision theory because
\begin{equation}
  \E{\ell_\epsilon(X, x)\cond \aset} = 
  1 - \Prob{\norm{X - x}\leq \epsilon \cond \aset}.
\end{equation}
Therefore, if $\hat x$ is a MAP estimate and $x$ is not, there exists an $\varepsilon\in\realspos$ such that, for all $\epsilon < \varepsilon$, the expected $\ell_\epsilon$ loss associated with $\hat x$ is no greater than that associated with $x$.
If two points are MAP estimates, then the ratio of the expected $\ell_\epsilon$ losses associated with them converges to unity.
Furtheremore, if $X$ is an $\reals^n$-valued random variable that admits a continuous density $p_X$, then for all $a,b\in\reals^n$ with $p_X(b)>0$
\begin{equation}
  \label{eq:ratio_densities}
  \lim_{\epsilon\downto 0}
  \frac{\Prob{\norm{X-a}\leq \epsilon}}{\Prob{\norm{X-b}\leq \epsilon}} =
  \frac{p_X(a)}{p_X(b)},
\end{equation}
so a point $x\in\reals^n$ is a MAP estimate if and only if it maximizes the density.
Consequently, Definition \ref{th:mode_defn} agrees with the \emph{de facto} definition of mode.

Having formalized the definition of mode for general random variables, we can now show how the Onsager--Machlup functional can be used to obtain the mode of the state paths of systems described by SDEs.

\section{Probability of tubes around state paths}
\label{sec:probability_tubes}
Let $X_t$ be an $\reals^n$-valued stochastic process representing the state of a dynamic system at the instant $t$ over the time interval $\tspace:=[0,T]$.
The evolution of $X$ is given by the SDE
\begin{equation}
  \label{eq:x_sde}
  \dd X_t = f(t,X_t)\,\dd t + G\,\dd W_t,
\end{equation}
where $f\colon\reals\times\reals^n\to\reals^n$ is the drift function, $G\in\reals^{n\times n}$ is the diffusion matrix and $W_t$ is an $n$-dimensional Wiener process which represents the process noise.

We now make some basic assumptions about the state process and the functions which define its dynamics.
A function will be said to be of class $C^n_b$ if it is $n$ times differentiable and all its derivatives up to degree $n$, including the zeroth, are bounded and continuous.

\needspace{2\baselineskip}
\begin{assum}[state process]
  \label{th:system_dynamics_assumptions}
  \hfill\vspace{-2ex}
  \begin{enumerate}[a.]
  \item\label{it:drift_c2b} The drift $f$ is continuous with respect to both arguments and is of class $C^2_b$ with respect to its second argument $x$, uniformly over its first argument $t$.
  \item\label{it:diffusion_rank} The diffusion matrix $G$ has full rank.
  \item\label{it:initial_state_wiener_process_independence} The initial state $X_0$ is independent of the Wiener process $W$.
  \item\label{it:initial_density} The initial state $X_0$ admits a continuous density $\nu(x_0)$ of class $C^0_b$
  \end{enumerate}
\end{assum}

Together, these assumptions allow the calculation of the asymptotic probability of $\epsilon$-tubes of the state process, also called $\epsilon$-sausages.
The $\epsilon$-tube around a given $\phi\colon\tspace\to\reals^n$ corresponds to all curves which lie within a distance $\epsilon$ of $\phi$, with respect to the supremum norm.
We denote by $\etube\in\events$ the event that the state process is contained within the $\epsilon$-tube around $\phi$, i.e.,
\begin{equation}
  \label{eq:tube_defn}
  \etube := \textstyle
  \setb{\omega\in\Omega}{
    \sup_{t\in\tspace} \norm{X_t(\omega) - \phi(t)} \leq \epsilon
  },
\end{equation}
where, for elements of $\reals^n$, $\norm{\cdot}$ denotes the Euclidean norm.
The probability of $\etube$ is also referred to as the sojourn probability around $\phi$.
The tube events define regions of the $\tspace\times\reals^n$ space, as illustrated in Fig.~\ref{fig:epsilon_tubes}, and an outcome $\omega\in\Omega$ belongs to the tube event if all of its pairs $\big(t,X_t(\omega)\big)$ lie in the shaded region.

Of special interest are the tubes around absolutely continuous functions with square-integrable derivatives, i.e., the Cameron--Martin space $\etaspace$ of the law of the Wiener process:
\begin{equation}
  \label{eq:cameron_martin_defn}
  \etaspace :=
  \setb{(x\colon\tspace\to\reals^n)}{
    x\text{ abs.\ continuous, }\dot x \in L^2(\tspace)
  }.
\end{equation}
For any such path $\phi\in\etaspace$, the Onsager--Machlup functional is defined as
\begin{align}
  \label{eq:onsager_machlup_functional}
  \textstyle
  J(\phi) := -\frac{1}{2}\int_0^T 
  \!\big[\normnm[Q]{\dot \phi(t) - f\big(t,\phi(t)\big)}^2 
  + \diverg f\big(t,\phi(t)\big)\big]\dd t,
\end{align}
where $Q:=(GG\trans)\inv$, $\norm[A]{x}:= \sqrt{x\trans Ax}$ and $\diverg f$ is the divergence of the drift vector field.
The importance of this functional is outlined in the theorem below.

\begin{thm}[\citet{fujita1982omf}]
  \label{th:onsager_machlup_limit}
  Consider a system \eqref{eq:x_sde} with state process satisfying Assumption \ref{th:system_dynamics_assumptions}.
  Then, for all $\phi,\varphi\in\etaspace$ with $\nu\fparen{\varphi(0)}>0$,
  \begin{equation}
    \label{eq:onsager_machlup_ratio}
    \lim_{\epsilon\downto 0}
    \frac{P(\etube)}{P(\etube[\varphi])}
    = \frac{
      \exp\fparen{\big.J(\phi)}\nu\fparen{\big.\phi(0)}
    }{
      \exp\fparen{\big.J(\varphi)}\nu\fparen{\big.\varphi(0)}
    }.
  \end{equation}
  where $J$ is the Onsager--Machlup functional defined in \eqref{eq:onsager_machlup_functional} and $\nu$ is the initial state density satisfying Assumption \ref{th:system_dynamics_assumptions}\ref{it:initial_density}.
\end{thm}

\begin{figure}
  \centering
  \inputtikzpicture{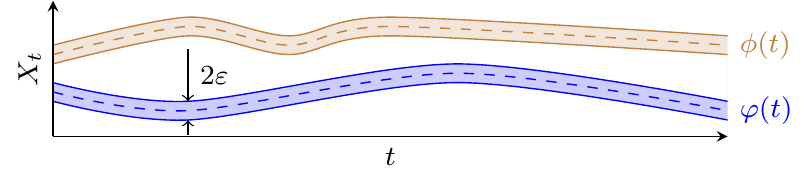}
  \caption{Illustration of the tubes corresponding to $\etube$ and $\etube[\varphi]$.}
  \label{fig:epsilon_tubes}
\end{figure}

Comparing \eqref{eq:onsager_machlup_ratio} to \eqref{eq:ratio_densities}, Thm.~\ref{th:onsager_machlup_limit} implies that, for the purpose of obtaining the mode, the product of the exponential of the Onsager--Machlup functional and the initial state density plays a role analogous to a probability density.
For this reason, these terms are sometimes referred to as a \emph{fictitious} density of state paths \citep{zeitouni1989omf}.
Also note that \eqref{eq:onsager_machlup_ratio} is valid for tubes with respect to other norms of function spaces besides the supremum norm, and that the limit does not depend on the norm used \citep{capitaine1995omf}.

For tubes around $\gamma\notin\etaspace$, outside the Cameron--Martin space, we assume that their probability decays much faster than that of tubes around $\varphi\in\etaspace$ that belong to the Cameron--Martin space, i.e., $\lim_{\epsilon\downto 0} \nicefrac{P(\etube[\gamma])}{P(\etube[\varphi])} = 0$.
This implies that the mode of the state path, according to Def.~\ref{th:mode_defn}, lies in $\etaspace$ and can be obtained by maximizing the denominator of the right-hand side of \eqref{eq:onsager_machlup_ratio}.
This assumption is implicit whenever the Onsager--Machlup functional is used to obtain the mode of state paths \citep{zeitouni1987map, aihara1999mem, aihara1999mln}.

Whenever the divergence of the drift vector field does not depend on the state $x$, then the last term of the integral in \eqref{eq:onsager_machlup_functional} can be dropped from the Onsager--Machlup functional, leading to the energy functional.
Even for systems with state-dependent drift divergence, the energy functional gives the asymptotic probability of $\epsilon$-tubes around the path of the Wiener process associated with the state path, as is shown in the next section.

\section{Probability of tubes around noise paths}

The Wiener process is also a diffusion process for which the Onsager--Machlup functional can be used for calculating the asymptotic probability of $\epsilon$-tubes.
For any $\eta\colon\tspace\to\reals^n$ with $\eta(0)=0$ and $x_0\in\reals^n$, let $\entube\in\events$ represent the event that the initial state and the Wiener process lie within an $\epsilon$ radius of $x_0$ and $\eta$, i.e.,
\begin{multline}
  \label{eq:noise_tube_defn}
  \entube :=
  \textstyle
  \{\omega\in\Omega\colon
  \sup_{t\in\tspace} \norm{W_t(\omega) - \eta(t)} \leq \epsilon,\\
  \norm{X_0(\omega) - x_0}\leq\epsilon\}.
\end{multline}
Applying a variant of Thm.~\ref{th:onsager_machlup_limit} which considers diffusions with a fixed initial condition \citep[cf.][]{fujita1982omf}, we have that, for any $\zeta,\eta\in\etaspace$ and $\chi_0,x_0\in\reals^n$ such that $\zeta(0)=\eta(0)=0$ and the initial state density $\nu(x_0)>0$,
\begin{equation}
  \label{eq:energy_functional_ratio}
  \lim_{\epsilon\downto 0}
  \frac{
    P(\entube[\zeta,\chi_0])
  }{
    P(\entube)
  }
  = \frac{
    \exp\fparen{-\frac{1}{2}\int_0^T\dot \zeta(t)\trans\dot\zeta(t)\dd t}
    \nu(\chi_0)
  }{
    \exp\fparen{-\frac{1}{2}\int_0^T\dot \eta(t)\trans\dot\eta(t)\dd t}
    \nu(x_0)
  }.
\end{equation}

To relate the sojourn probability of state paths and noise paths, we now show that the event corresponding to each noise path tube is contained in an event corresponding to a state path tube.

\begin{lem}
  \label{th:noise_state_tube_equivalency}
  If the system satisfies Assumption \ref{th:system_dynamics_assumptions}, then there exists a constant $c_1\in\reals$ such that $\entube[\eta,\phi(0)]\subset\mathcal{B}^{c_1\epsilon}_\phi$ for all $\epsilon\in\reals$ and $\eta,\phi\in\etaspace$ such that $\eta(0)=0$ and
  \begin{equation}
    \label{eq:associated_ode}
    \dot\phi(t) = f\fparen{\big. t, \phi(t)} + G\dot\eta(t).
  \end{equation}
\end{lem}
\begin{proof}
  The SDE \eqref{eq:x_sde} can be written in the integral form as
  \begin{equation}
    \label{eq:sde_integral_form}
    \textstyle
    X_t = X_0 + \int_0^tf(t,X_t)\,\dd t + GW_t.
  \end{equation}
  Writing the associated ordinary differential equation (ODE)
  \eqref{eq:associated_ode} in the same form and taking the difference, we
  have that
  \begin{multline}
    \norm{X_t - \phi(t)} \leq \norm{X_0 - \phi(0)}
    + \norm{GW_t -G\eta(t)}\\
    \textstyle
    + \int_0^t\norm{f(t, X_t) - f\fparen{\big. t, \phi(t)}} \dd t.
  \end{multline}
  Since the Jacobian matrix of $f$ is assumed to be bounded, $f$ is Lipschitz continuous, i.e., there exist $K\in\realspos$ such that
  \begin{equation}
    \label{eq:lipschitz_condition}
    \norm{f(t,a) -f(t, b)} \leq K\norm{a-b} 
  \quad \forall t\in\tspace \text{ and }a,b\in\reals^n.
  \end{equation}
  Defining $c_2:=1+\norm{G}$, where the norm of $G$ is induced by the Euclidean norm, we have that, for all $\omega\in\entube[\eta,\phi(0)]$,
  \begin{equation}
    \textstyle
    \norm{X_t - \phi(t)} \leq c_2\epsilon
    + \int_0^t K\norm{X_t - \phi(t)} \dd t.
  \end{equation}
  Using the Bellman--Grönwall inequality we then have that
  \begin{equation}
    \norm{X_t - \phi(t)} \leq \epsilon c_2 \exp(TK).
    \qedhere
  \end{equation}
\end{proof}

Lem.~\ref{th:noise_state_tube_equivalency} means that every state path in $\etaspace$ is associated with a noise path in $\etaspace$ and a initial state through the associated ODE \eqref{eq:associated_ode}.
Rewriting \eqref{eq:energy_functional_ratio} in terms of the state path $\phi$ associated with the noise path $\eta$, we note that
\begin{equation}
  -\tfrac{1}{2}\dot\eta(t)\trans\dot\eta(t) = 
  -\tfrac{1}{2}
  \big\lVert \dot\phi(t) - f\fparen{\big. t, \phi(t)}\!\big\rVert_Q^2,
\end{equation}
which leads to the energy functional:
\begin{equation}
  \label{eq:energy_functional_defn}
  \textstyle
  J_\energy(\phi) := -\frac{1}{2}\int_0^T 
  \big\lVert\dot\phi(t) - f\fparen{\big. t, \phi(t)}\!\big\rVert_Q^2\dd t,
\end{equation}
so called because it corresponds to the energy of the noise signal associated with each state path.
The energy functional is the Onsager--Machlup functional without the drift divergence term, which accounts for the amplification of noise by the system dynamics.

Next, we show how to incorporate the measurement likelihood so that the asymptotic \emph{posterior} sojourn probability can be obtained and, with it, the MAP estimator.

\section{Continuous-discrete MAP estimation}
\label{sec:continuous-discrete_map}
Let the finite set $\tmeas\subset\tspace$ be the time instants where measurements are taken.
In addition, let the $\reals^{m_t}$-valued random variables $Y_t$ denote the measurements taken at each $t\in\tmeas$, with the dimension of the measurement vector $m_t$ possibly varying across different instants.
We then make the following assumptions.

\begin{assum}[measurements]
  \label{th:measurement_assum}
  \hfill\vspace{-2ex}
  \begin{enumerate}[a.]
  \item \label{it:measurement_conditional_independence} Each measurement $Y_t$ is conditionally independent, given the state $X_t$ at the measurement instant, of all the states $X_\tau$ and measurements $Y_\tau$ at the remaining time instants $\tau\neq t$.
  \item Each measurement $Y_t$ admits a conditional density $\psi_t(y|x)$, given the state $X_t$, which is of class $C^0_b$ with respect to its second argument.
  \item All measured values $y_t$ have a nonzero predictive prior density, i.e., $\E{\psi_t(y_t|X_t)}>0$.
  \end{enumerate}
\end{assum}

Given a sequence $y_t$ of measurement values, we wish to obtain the mode of the state path conditioned on
$\yset := \setb{\omega\in\Omega}{Y_t(\omega) = y_t,\forall t\in\tmeas}$.
We then have the following result.

\begin{thm}
  \label{th:posterior_tube_ratio}
  Let $\phi,\varphi\in\etaspace$ with $\psi_t\fparen{\big.y_t\cond\varphi(t)}>0$ for all $t\in\tmeas$.
  In addition, let $\aset_1^\epsilon$ and $\aset_2^\epsilon$ be two families of events measurable with respect to the $\sigma$-algebra generated by $X$ such that $\aset_1^\epsilon\subset\etube$, $\aset_2^\epsilon\subset\etube[\varphi]$ and $\nicefrac{P(\aset_1^\epsilon)}{P(\aset_2^\epsilon)}$ converges as $\epsilon\downto 0$.
  Then, for a system with state process satisfying Assumption \ref{th:system_dynamics_assumptions} and measurements satisfying Assumption \ref{th:measurement_assum},
  \begin{equation}
    \label{eq:posterior_tube_ratio}
    \lim_{\epsilon\downto 0}
    \frac{P(\aset_1^\epsilon|\yset)}{P(\aset_2^\epsilon|\yset)}
    = 
    \frac{
      \prod_{t\in\tmeas} \psi_t\fparen{\big.y_t\cond \phi(t)}
    }{
      \prod_{t\in\tmeas} \psi_t\fparen{\big.y_t\cond \varphi(t)}
    }
    \lim_{\epsilon\downto 0}\frac{P(\aset_1^\epsilon)}{P(\aset_2^\epsilon)}.
  \end{equation}
\end{thm}
\begin{proof}
  Using Bayes' rule, the posterior probability of any event $\aset$ measurable with respect to the $\sigma$-algebra generated by $X$ is
\begin{equation}
  \label{eq:posterior_xevent_probability}
  P(\aset|\yset) =
  \frac{
    P(\aset) \prod_{t\in\tmeas}\E{\psi_t(y_t|X_t)\cond\aset}
  }{\prod_{t\in\tmeas}\E{\psi_t(y_t|X_t)}}.
\end{equation}
Due to the continuity of the measurement likelihood $\psi_t$ with respect to its second argument, we have that for all $t\in\tmeas$,
\begin{equation}
  \label{eq:likelihood_limit}
  \lim_{\epsilon\downto 0}\E{\psi_t(y_t|X_t)\cond\aset_1^\epsilon} =
  \psi_t\fparen{y_t\cond \phi(t)\big.},
\end{equation}
and analogously for $\aset_2^\epsilon$ and $\varphi$.
Substituting \eqref{eq:posterior_xevent_probability} and \eqref{eq:likelihood_limit} into the left-hand side of \eqref{eq:posterior_tube_ratio} we obtain the result on the right-hand side of the equation.
\end{proof}

The MAP state path estimation problem then follows from Theorem~\ref{th:posterior_tube_ratio}, by setting $\aset^\epsilon = \etube$.
We will use the logarithm of the asymptotic sojourn probability because it is more tractable numerically and simplifies some convergence proofs.
As some densities can be equal to zero, the resulting merit functions will take values on the extended real number line $\ereals:=\reals\cup\{-\infty,\infty\}$.

\begin{lem}
  \label{th:cd_map_problem}
  The continuous-discrete MAP state paths are curves $\phi\in\etaspace$ that maximize the MAP merit $H\colon\etaspace\to\ereals$ given by
  \mathtoolsset{showonlyrefs=false}
  \begin{align}
    \label{eq:posterior_map_merit}
    \textstyle
    H(\phi) := 
    J(\phi) + \ln\nu\fparen{\big.\phi(0)} +
    \sum_{t\in\tmeas} \ln\psi_t\fparen{\big.y_t\cond \phi(t)}.
  \end{align}
  \mathtoolsset{showonlyrefs=true}
\end{lem}

If instead $\aset^\epsilon = \entube$, Theorem~\ref{th:posterior_tube_ratio} leads to the minimum energy estimation problem.

\begin{lem}
  \label{th:minimum_energy_problem}
  The minimum energy state paths, which are associated with the MAP noise paths, are curves $\phi\in\etaspace$ that maximize the energy merit $H_\energy\colon\etaspace\to\ereals$ given by
  \begin{align}
    \label{eq:posterior_energy_merit}
    \textstyle
    H_\energy(\phi) := 
    J_\energy(\phi) + \ln\nu\fparen{\big.\phi(0)} +
    \sum_{t\in\tmeas} \ln\psi_t\fparen{\big.y_t\cond \phi(t)}.
  \end{align}
\end{lem}

We now present the discretization schemes and their corresponding discretized MAP state path estimation problems.

\section{Discretized MAP state path estimation}
\label{seq:discretized_map}
To use discrete-time MAP state path estimation methods in systems with dynamics described by SDEs, the state transition density between points in a partition of the experiment's time interval must be known.
As the transition densities of SDEs do not have closed-form solutions for general nonlinear systems, approximations must be used.
In this work, we consider approximations used in the numerical solution of SDEs.

We start with an ordered partition
$\partition := \{t_k\}_{k=0}^N\subset\tspace$, which includes the endpoints $t_0=0$, $t_N=T$ and measurement times $\tspace_m\subset\partition$.
The discretization schemes then represent the evolution of the state at the partition points by difference equations.
For brevity when working with partition points, we will use the following notation when unambiguous: $X_{t_k}$ will be shortened to  $X_k$ and sequences like $X_{t_0},\dotsc,X_{t_N}$ will be shortened to $X_{0:N}$.
The sequence $X_{0:N}$ will be denoted the discretized state path.

Change of variables is used to obtain the discretized state path densities.
For this we need the well-known theorem below.

\begin{thm}[{\citet[Prop.~9.1]{georgii2008s}}]
  \label{th:change_densities}
  Let $A$ be an $\reals^n$-valued random variable admitting a probability density.
  If $q\colon\reals^n\to\reals^n$ is a $C^1$ diffeomorphism and the $\reals^n$-valued random variable $B$ is defined as $B := q(A)$, then $B$ admits a probability density, which is given by
  \begin{equation}
    \label{eq:change_densities}
    \pr[B]{b} = \abs{\det \nabla q\inv(b)}\pr[A]{q\inv(b)},
  \end{equation}
  where $\det\nabla q\inv(b)$ is the Jacobian determinant of the inverse of $q$, evaluated at $b$.
\end{thm}

In \eqref{eq:change_densities}, the Jacobian determinant represents how the diffeomorphism changes the volume element at each point.
We now present each discretization and the corresponding discretized MAP state path estimation problems.

\subsection{Euler scheme}
The stochastic Euler scheme \citep[Sec.~9.1]{kloeden1992nss}, also known as the Euler--Maruyama scheme, is one of the simplest SDE discretization schemes.
It approximates the evolution of the states by
\begin{equation}
  \label{eq:euler_maruyama_scheme}
  X_{k+1} = X_k + f(t_k,X_k)\delta_k + G\Delta W_k,
\end{equation}
where $\delta_k:= t_{k+1} - t_k$ and $\Delta W_k:=W_{k+1} - W_k$.
Recall that the increments of the Wiener process are, by definition, independent and normally distributed with zero mean and covariance $\ident\delta_k$, i.e., $\Delta W_k \sim \normald{0,\ident\delta_k}$.

From \eqref{eq:euler_maruyama_scheme} and the assumptions that $f$ is of class $C^2_b$ and $G$ is invertible, it can be seen that $X_{0:N}$ can be obtained from $X_0,\Delta W_{0:N-1}$ by a diffeomorphism.
Applying Thm.~\ref{th:change_densities}, the joint prior density of the state path over the discretization points is given by
\begin{equation}
  \label{eq:euler_state_path_prior}
  p(x_{0:N}) = \frac{
    \nu(x_0) \exp\fparen{\big.R(x_{0:N})}
  }{
    \prod_{k=0}^{N-1}\abs{\det G}\sqrt{(2\pi)^n\delta_k}
  },
\end{equation}
where $R$ is the Euler energy functional, defined as 
\begin{equation}
  \label{eq:euler_energy_functional}
  R(x_{0:N}) := 
  -\frac{1}{2} \sum_{k=0}^{N-1} \delta_k
  \norm[Q]{\frac{\Delta x_k}{\delta_k} - f(t_k,x_k)}^2.
\end{equation}
Assumption \ref{th:measurement_assum}\ref{it:measurement_conditional_independence} implies that the joint measurement likelihood, i.e., the joint density of all measurements, conditioned on the state path, is given by
\begin{equation}
  \label{eq:meas_likelihood}
  \textstyle
  p(\yset|x_{0:N}) = \prod_{t\in\tmeas} \psi_t(y_t|x_t).
\end{equation}
Applying Bayes' rule we have that, for fixed $y_t$, the posterior state path density can be obtained from \eqref{eq:euler_state_path_prior} and \eqref{eq:meas_likelihood}:
\begin{equation}
  \label{eq:posterior_state_path_density}
   p(x_{0:N}|\yset)\propto p(\yset|x_{0:N}) p(x_{0:N}).
\end{equation}
Taking the logarithm of \eqref{eq:posterior_state_path_density} and removing the constant terms, which do not influence the location of maxima, we obtain the merit function for MAP estimation.

\begin{lem}
  \label{th:euler_map_problem}
  The discretized MAP state paths under the Euler scheme are the vectors $x_{0:N}\in\reals^{Nn+n}$ which maximize the Euler merit
  \begin{align}
    \label{eq:posterior_euler_merit}
    \textstyle
    S(x_{0:N}) := 
    R(x_{0:N}) + \ln\nu\fparen{\big.x_0} +
    \sum_{t\in\tmeas} \ln\psi_t\fparen{\big.y_t\cond x_t}.
  \end{align}
\end{lem}

\subsection{Trapezoidal scheme}
The trapezoidal scheme \citep[p.~500]{kloeden1992nss} represents the SDE over the partition points by an implicit difference equation:
\begin{align}
  \label{eq:trapezoidal_scheme}
  X_{k+1} = X_k + \tfrac{\delta_k}{2}\brackt{\big.f(t_k,X_k) + f(t_{k+1},X_{k+1})}
  + G\Delta W_k.
\end{align}

Since the drift $f$ is assumed to be of class $C^2_b$, for sufficiently small $\delta_k$ \eqref{eq:trapezoidal_scheme} is a contraction mapping and always has a unique solution.
We shall assume that the partition is sufficiently fine such that this holds.

As with the Euler scheme, applying Thm.~\ref{th:change_densities} we obtain the joint discretized state path prior density:
\begin{equation}
  \label{eq:trapezoidal_prior}
  p(x_{0:N}) = \frac{
    \nu(x_0) \exp\fparen{\big.U(x_{0:N})}
  }{
    \prod_{k=0}^{N-1}\abs{\det G}\sqrt{(2\pi)^n\delta_k}
  },
\end{equation}
where $U$ is the trapezoidal Onsager--Machlup functional defined as
\begin{multline}
  \textstyle
  U(x_{0:N}) :=
  \sum_{k=0}^{N-1} \ln\det
  \fparen{\big.\ident -\tfrac{1}{2}\nabla_xf(t_{k+1},x_{k+1})\delta_{k}}
  \\
  \textstyle
  -\frac{1}{2}
  \sum_{k=0}^{N-1} \delta_k\normnm[Q]{
    \frac{\Delta x_k}{\delta_k} - 
    \frac{1}{2}[f(t_k,x_k) + f(t_{k+1},x_{k+1})]
  }^2
\end{multline}
with $\nabla_x f(t,a)$ being the Jacobian matrix of $f$ with respect to $x$, evaluated at $t,a$.
Using equations \eqref{eq:meas_likelihood}, \eqref{eq:posterior_state_path_density} and \eqref{eq:trapezoidal_prior} and taking their logarithm, we obtain the discretized MAP state path estimation problem using the trapezoidal scheme.
\begin{lem}
  \label{th:trapezoidal_map_problem}
  The discretized MAP state paths under the trapezoidal scheme are the vectors $x_{0:N}\in\reals^{Nn+n}$ which maximize the trapezoidal merit
  \begin{align}
    \label{eq:trapezoidal_merit}
    \textstyle
    V(x_{0:N}) := 
    U(x_{0:N}) + \ln\nu\fparen{\big.x_0} +
    \sum_{t\in\tmeas} \ln\psi_t\fparen{\big.y_t\cond x_t}.
  \end{align}
\end{lem}

Unlike the Euler scheme, in the trapezoidal the state $X_{k+1}$ depends nonlinearly on the Wiener process increment $\Delta W_k$, due to the implicit nature of \eqref{eq:trapezoidal_scheme}.
This means that the prior density of a state path does not depend only on the density of the associated noise, but also on the change of the volume element, represented by the Jacobian determinant of the drift function.
This extra term represents how much the system dynamics amplifies or attenuates the noise around each path, akin to the drift divergence in the Onsager--Machlup functional.
We note that this term has a numerical and conceptual connection to Lyapunov exponents, whose sum quantifies the average change in hypervolume due to pertubations of a path of a noise-free discrete-time map.

We will now present the theory of hypographical convergence to relate the discretized and continuous-discrete MAP estimation problems.

\section{Limit of the discretized problems}
\label{sec:limit_discretized_problems}
We now consider the discretized problems over a nested sequence of partitions $\{\partition_i\}_{i=1}^\infty$.
The last index of each partition will be denoted $N_i$, their elements by $t_{ki}$ and the time increments by $\delta_{ki}:=t_{k+1,i}-t_{ki}$.
In addition, $\mesh_i$ will denote the mesh (or norm) of the partition $\partition_i$, i.e., the largest distance between consecutive points
\begin{equation}
  \label{eq:partition_norm_defn}
  \textstyle
  \mesh_i := \max_{\,0\leq k<N_i} \delta_{ki}.
\end{equation}
We consider partition sequences which satisfy
\begin{align}
  \label{eq:partition_requirements}
  \partition_i&\subset\partition_{i+1},&
  \textstyle\lim_{i\to\infty}\mesh_i & = 0, &
  \mesh_iN_i < c_3,
\end{align}
where $c_3>0$ is an upper bound.

To relate the various discretizations and the continuous-discrete problems, we must formulate all problems on the same space.
We do this by representing the discretized state paths $\phi_i\colon\tspace\to\reals^n$ as the piecewise linear interpolants of given $x_{0:N_i}$.
We also denote by $\etaspace_i\subset\etaspace$ the space of all piecewise linear functions from $\tspace$ to $\reals^n$ with breaks along $\partition_i$.
The discretized merit functions for each partition will be denoted by $S_i$ and $V_i$, and are accordingly extended for inputs from $\etaspace_i$.
The same goes for the Euler energy functional $R_i$ and the trapezoidal Onsager--Machlup functional $U_i$.

The Cameron--Martin space $\etaspace$ is a reproducing kernel Hilbert space.
It can be endowed with the following inner product and its associated norm:
\begin{equation}
  \label{eq:cameron_martin_inner_product}
  \textstyle
  \innerp[\etaspace]{a,b}:= a(0)\trans b(0) + 
  \int_0^T \dot a(t)\trans\dot b(t) \,\dd t.
\end{equation}
This norm for $\etaspace$ dominates the supremum norm \citep[Prop.~4.1]{daniel1969amf}, i.e., there exists a constant $c_4\in\reals$ such that, for all $\phi\in\etaspace$,
\begin{equation}
  \label{eq:supremum_norm_dominated_sobolev_norm}
  \textstyle
  \sup_{t\in\tspace}\norm{\phi(t)} \leq c_4 
  \norm[\etaspace]{\phi}.
\end{equation}

We use hypographical convergence, or hypo-convergence, to relate the discretized problems and their continuous-discrete counterparts.
Hypo-convergence is a powerful tool because when a series of optimization problems hypo-converge, then any cluster point of their solutions solves the limit problem \citep[Thm.~3.3.3]{polak1997oac}.
Hypographical limits are formally defined as the limit of the hypographs of a series of optimization problems but, in metric spaces, the following definition is equivalent and more convenient to use.

\begin{thm}[{adapted from \citet[Thm.~3.3.2]{polak1997oac}}]
  Let $\etaspace$ be a normed linear space and $\{\etaspace_i\}_{i=1}^\infty$ be a nested family of finite-dimensional subspaces whose limit $\etaspace_\infty$ is dense in $\etaspace$, i.e.,
  \begin{align}
    \etaspace_i\subset\etaspace_{i+1}&\subset\etaspace,&
    \etaspace_\infty &:= \cup_{i=1}^\infty \etaspace_i, &
    \closure{\etaspace_\infty} = \etaspace,
  \end{align}
  where $\closure$ stands for topological closure.
  Additionally, let $H\colon\etaspace\to\ereals$ be the upper semi-continuous merit function of a maximization problem and $H_i\colon\etaspace_i\to\ereals$ be a family of upper semi-continuous merit functions of a family of maximization problems.
  A sufficient condition for the family of maximization problems to hypo-converge to the maximization of $H$ is that for every convergent infinite sequence $\{\phi_i\}_{i\in\ispace}$ where $\phi_i\in\etaspace_i$ and $\phi_i\to\phi$, then
  \begin{equation}
    \textstyle
    \limsup_{i\to\infty} H_i(\phi_i) = H(\phi).
  \end{equation}
\end{thm}

We now present a few lemmas that will help prove hypo-convergence.

\begin{lem}
  \label{th:convergence_interpolates}
  Let $q\colon\tspace\times\reals^n\to\reals^o$ be a continuous function and $\{\phi_i\}_{i\in\ispace}$ be a convergent infinite sequence where $\phi_i\in\etaspace_i$ and $\phi_i\to\phi$.
  Additionally, let $g_i\colon\tspace\to\reals^n$ be the piecewise constant interpolations of $q\big(t,\phi_i(t)\big)$ with breaks along the partitions $\partition_i$ and define $g(t):=q\fparen{\big.t,\phi(t)}$.
  Then $g_i$ and $g$ are bounded and
  \begin{equation}
    \label{eq:difference_between_interpolates}
    \textstyle
    \lim_{i\to\infty}\sup_{t\in\tspace}\norm{g(t)-g_i(t)} =0.
  \end{equation}
\end{lem}
\begin{proof}
  From \eqref{eq:supremum_norm_dominated_sobolev_norm} and the fact that the $\phi_i$ sequence is convergent with respect to the $\etaspace$-norm, it follows that $\sup_{i,t}\norm{\phi_i(t)}<\infty$.
  We can then restrict our analysis of $q$ to a compact subset of its domain, where it is bounded and uniformly continuous.
  Let $\rho_q$ denote a modulus of continuity of $q$, i.e., for all $t,\tau\in\tspace$ and $a,b$ in the compact subset of $\reals^n$ being analysed,
  \begin{equation}
    \label{eq:modulus_of_continuity_defn}
    \norm{q(t,a)-q(\tau,b)} \leq \rho_q\fparen{\big.\abs{t-\tau},\norm{a-b}}.
  \end{equation}
  Similarly, let $\rho_\phi$ denote the modulus of continuity of $\phi$ and define $e_i:=c_4\norm[\etaspace]{\phi_i-\phi} + \rho_\phi(\mesh_i)$.
  Then $e_i\to 0$ and 
  \begin{equation}
    \label{eq:max_difference_between_phis}
    \textstyle
    \sup_{t,\tau\in[t_{ki}, t_{k+1,i}]} \normnm{\phi_i(t)-\phi(\tau)} \leq e_i.
  \end{equation}
  Combining \eqref{eq:modulus_of_continuity_defn} and \eqref{eq:max_difference_between_phis} we have that
  \belowdisplayskip=-12pt
  \begin{align*}
    \textstyle
    \sup_{t\in\tspace}\norm{g_i(t)-g(t)} &\leq \rho_q(\mesh_i, e_i),
    & \rho_q(\mesh_i, e_i)\to 0.
  \end{align*}
\end{proof}

We now prove that, for any convergent sequence of piecewise linear functions over the partitions, the discretized merits converge to the continuous-discrete merits.
The $L^2$ space for $\reals^n$-valued functions will be denoted by $L^2_n$, and corresponds to the direct sum of $n$ copies of the standard $L^2$ space, with its inner product given by
\begin{equation}
  \label{eq:l2n_norm}
  \textstyle
  \innerp[L^2_n]{a,b} := \int_0^T a(t)\trans b(t) \,\dd t.
\end{equation}

\begin{lem}
  \label{th:euler_merit_convergence}
  Let $\{\phi_i\}_{i\in\ispace}$ be a convergent infinite sequence where $\phi_i\in\etaspace_i$ and $\phi_i\to\phi$.
  Then
  \begin{equation}
    \textstyle
    \lim_{i\to\infty}S_i(\phi_i) = H_\energy(\phi),
  \end{equation}
  where $H_\energy$ is defined in \eqref{eq:posterior_energy_merit} and $S_i$ is given by \eqref{eq:posterior_euler_merit} for partition $\partition_i$.
\end{lem}
\begin{proof}
  Since by \eqref{eq:supremum_norm_dominated_sobolev_norm} convergence in $\etaspace$ implies uniform convergence and the initial state density $\nu$ and measurement likelihoods $\psi_t$ are continuous, we have that
  \begin{align}
    \nu\fparen{\big.\phi_i(0)} &\to \nu\fparen{\big.\phi(0)}, &
    \psi_t\fparen{\big.\phi_i(t)} &\to\psi_t\fparen{\big.\phi(t)},
  \end{align}
  so, as the logarithm is also continuous in $\ereals$, it only remains for us to prove that $R_i(\phi_i)\to J_\energy(\phi)$.

  Let $g_i$ be the piecewise constant interpolation of $f\fparen{\big.t,\phi_i(t)}$ with breaks across $\partition_i$ and with the left endpoint of each piece as the interpolation point.
  In addition, define $g(t):=f\fparen{\big.t,\phi(t)}$.
  Then, using the bound $\norm[Q]{x}\leq\lVert Q\rVert\norm{x}$ and the fact that $\dot\phi_i$ is piecewise constant, from the definitions \eqref{eq:energy_functional_defn} and \eqref{eq:euler_energy_functional} of the functionals we have that
  \begin{equation}
    \abs{R_i(\phi_i) - J_\energy(\phi)}
    \leq \lVert Q\rVert
    \paren{
      \normnm[L^2_n]{\dot\phi_i - g_i}^2
      - \normnm[L^2_n]{\dot\phi - g}^2
    }.
  \end{equation}
  As the norm and exponentiation are continuous functions, it suffices to show that $\phi_i\to\phi$ and $g_i\to g$ with respect to the $L^2_n$ norm.
  The former follows trivially from the fact that $\norm[L^2_n]{\dot \varphi}\leq\norm[\etaspace]{\varphi}$ for all $\varphi\in\etaspace$.
  The latter follows from Lem.~\ref{th:convergence_interpolates} and the fact that uniform convergence implies convergence in $L^2$ for finite measure spaces.
\end{proof}

\begin{lem}
  \label{th:trapezoidal_merit_convergence}
  Let $\{\phi_i\}_{i\in\ispace}$ be a convergent infinite sequence where $\phi_i\in\etaspace_i$ and $\phi_i\to\phi$.
  Then
  \begin{equation}
    \textstyle
    \lim_{i\to\infty}V_i(\phi_i) = H(\phi),
  \end{equation}
  where $H$ is defined in \eqref{eq:posterior_map_merit} and $V_i$ is given by \eqref{eq:trapezoidal_merit} for partition $\partition_i$.
\end{lem}
\begin{proof}
  As in the proof of Lem.~\ref{th:euler_merit_convergence}, the convergence of the terms corresponding to $\nu$ and $\psi_t$ follows trivially, so we will concentrate on proving that $U_i(\phi_i)\to J(\phi)$.
  Define
  \begin{multline}
    \textstyle
    e_i:= \sum_{k=0}^{N-1} \ln\det
    \fparen{\big.\ident
      -\tfrac{1}{2}\nabla_xf(t_{k+1},\phi_i(t_{k+1}))\delta_{k}
    } \\
    \textstyle
    - \frac{1}{2}\int_0^T\diverg f\fparen{\big.t, \phi(t)} \dd t.
  \end{multline}
  We now show that $e_i\to 0$.
  
  First, recall that for any nonsingular matrix $A\in\reals^{n\times n}$, $\ln\det A=\tr\ln A$, where $\ln A$ is any matrix logarithm of $A$.
  In addition, for the principal logarithm we have that, for all matrices $A$ with spectral radius smaller than unity, $\ln(\ident + A) = \sum_{j=1}^\infty (-1)^{j+1}A^j/j$ \citep[Sec.~5.2]{higham2010cmf}.
  Since the drift $f$ is assumed to be of class $C^2_b$, we have that for small enought $\mesh_i$ (or, equivalently, sufficiently large $i$) $\nabla_xf(t,x)\delta_{ki}$ always has spectral radius smaller than unity.
  This implies that there exists a constant $c_{10}$ such that, for all $i$ above a certain limit, $x\in\reals^n$ and $t\in\tspace$, 
  \begin{align}
    \abs{
      \ln \det \fparen{\big.\ident - \tfrac{1}{2}\nabla_xf(t,x)\delta_{ki}} 
      - \tfrac{1}{2}\diverg f(t,x)\delta_{ki}
    } \leq c_{10}\mesh_i^2.
  \end{align}
  Let $d_i$ denote the piecewise constant interpolation of $\diverg f\fparen{\big.t,\phi_i(t)}$ with breaks across $\partition_i$ and with the right endpoint of each piece as the interpolation point.
  Likewise, define $d(t):= \diverg f\fparen{\big.t,\phi(t)}$.
  We then have that, for all sufficiently large $i$,
  \begin{equation}
    \textstyle
    \abs{e_i} \leq \frac{1}{2}\int_0^T \abs{d(t) - d_i(t)} \dd t 
    + c_{10} N_i\mesh_i^2.
  \end{equation}
  From Lem.~\ref{th:convergence_interpolates} and the fact that uniform convergence implies convergence in $L^1$ in finite measure spaces we have that $d_i\to d$ in $L^1$.
  Also, from \eqref{eq:partition_requirements} we have that $N_i\mesh_i$ is bounded, so $e_i\to 0$.

  We now show that the quadratic terms of the merit functions converge.
  Let $g_i$ and $h_i$ denote piecewise constant interpolations of $f\fparen{\big.t,\phi_i(t)}$ with breaks across $\partition_i$; $g_i$ with the left endpoint of each piece as the interpolation point and $h_i$ with the right endpoint.
  It suffices to prove that
  \begin{equation}
    \textstyle
    \lim_{i\to\infty} \normnm[L^2_n]{\dot\phi_i - \frac{1}{2}(g_i+h_i)}^2
    - \normnm[L^2_n]{\dot\phi - g}^2 = 0.
  \end{equation}
  From Lem.~\ref{th:convergence_interpolates} we have that $g_i\to g$ and $h_i\to g$ in $L^2_n$.
  Furthermore, $\phi_i\to\phi$, so due to continuity of the norm and exponentiation $V_i(\phi_i)\to H(\phi)$.
\end{proof}

\begin{lem}
  \label{th:denseness_etaspace_infty}
  The space $\etaspace_\infty$ is dense in $\etaspace$.
\end{lem}
\begin{proof}
  Denote by $\sspace_i$ the space of piecewise constant functions from $\tspace$ to $\reals^n$ with breaks along $\partition_i$.
  The $\sspace_i$ are spaces of order 1 splines and their union is dense in the space of continuous functions equipped with the supremum norm \citep[p.~147]{deboor2001pgs}.
  Since continuous functions are dense in $L^2_n$ and denseness is a transitive relation, $\sspace_i$ is dense in $L^2_n$.
  Thus, for any $\phi\in\etaspace$ and $\epsilon\in\realsnonneg$, there exists an $i\in\naturals$ and $u\in\sspace_i$ such that $\normnm[L^2_n]{u-\dot\phi}<\epsilon$.
  If we then define $\phi_i(t):= \phi(0) +\int_0^t\dot\phi(\tau)\dd\tau$ we have that $\phi_i\in\etaspace_i$ and $\norm[\etaspace]{\phi_i-\phi}<\epsilon$.
\end{proof}

We are now ready to prove one of the main results of this paper.

\begin{thm}
  \label{th:main_theorem}
  The discretized MAP state path estimation with the trapezoidal scheme hypo-converges to the continuous-discrete MAP state path estimation; and the discretized MAP state path estimation with the Euler scheme hypo-converges to the minimum energy estimation.
\end{thm}
\begin{proof}
  From Lem.~\ref{th:denseness_etaspace_infty} we know that $\etaspace_\infty$ is dense in $\etaspace$.
  From Lemmas \ref{th:euler_merit_convergence} and \ref{th:trapezoidal_merit_convergence} we then have that the merits converge for any convergent sequence of $\phi_i$.
\end{proof}

In the sequence, we present simulated examples to illustrate an application of robust estimation under the presence of outliers using both the MAP and minimum energy estimators.

\section{Simulated Examples}
\label{sec:simulated_experiments}

The stochastic Van der Pol oscillator is a nonlinear system commonly used to compare state estimators.
This system, discretized with the Euler scheme, was used by \citet{aravkin2011llr, aravkin2012rtf} for a comparison of robust discrete-time estimators.
Denoting its states  by $x := [u,\; v]\trans$, its dynamics are defined by
\begin{align}
  \label{eq:vdp_sde}
  f(t, x) &:=
  \begin{bmatrix}
    v \\ -u + 2(1 - u^2)v
  \end{bmatrix}, &
  G &:=
  \begin{bmatrix*}[c]
    \np{0.1} & \np{0} \\ \np{0} & \np{0.1}
  \end{bmatrix*}.
\end{align}

This system has a nonlinear drift divergence, given by
$\diverg f(t, x) = 2(1 - u^2)$, which is higher near the origin.
Consequentely, the MAP merit will favor paths with larger $u$ while the energy merit will not make such distinctions.
To highlight this difference between the estimators, the initial state was chosen around the origin, drawn from the normal distribution $\normald{[0,\; 0]\trans\!, \tfrac{1}{100}\ident\big.}$.

The system was simulated using the strong explicit order 1.5 scheme  \citep[Sec.~11.2]{kloeden1992nss} with time step \np{5e-4} and experiment duration $T=\np{16}$.
The measurements were generated according to the Gaussian mixture distribution
\begin{equation}
  \label{eq:outlier_distribution}
  Y_t|X_t\sim 
  (1 - p_o)\normald{U_t, \sigma_y^2 \big.} +
  p_o\normald{U_t, \sigma_o^2 \big.},
\end{equation}
where $\sigma_y = 0.5$ is the regular measurements' standard deviation, $\sigma_o$ is the outliers' standard deviation and $p_o$ is the outlier probability, the latter two of which were varied across experiments.
The measurements were taken at time intervals of \np{0.1}, i.e., $\tmeas = \{0, 0.1, \dotsc, 16\}$.

\begin{figure}[tb]
  \centering
  \inputtikzpicture{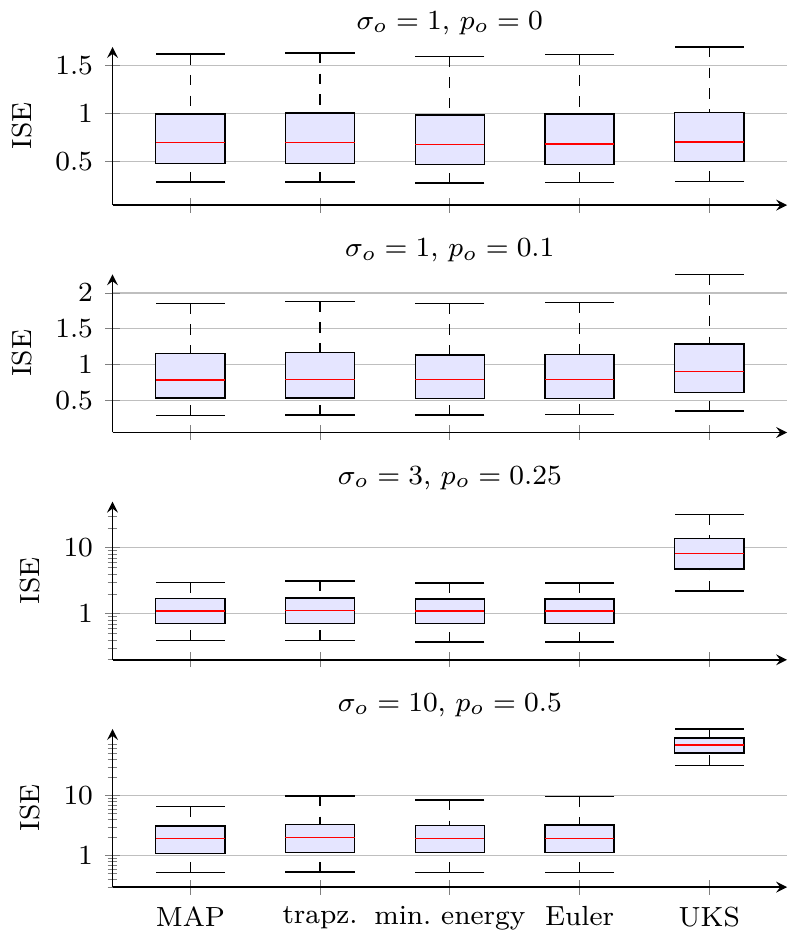}
  \caption{
    Box and whisker plot of the integrated squared error for the robust estimation.
    The whiskers indicate the 5th and 95th percentiles.
  }
  \label{fig:robust_estimation}
\end{figure}

Student's $t$-distribution was used as the measurement likelihood for the MAP, minimum energy and discretized estimators to confer robustness against outliers.
The expression of the log-likelihood used in the estimation is
\begin{equation}
  \label{eq:robust_meas_likelihood}
  \ln \psi_t(y|x) = 
  -\tfrac{5}{2}\ln\paren{1 + \tfrac{(u - y)^2}{4\sigma_y^2}},
\end{equation}
which corresponds to the $t$-distribution with 4 degrees of freedom and scaling $\sigma_y$.
The estimates were also compared against the unscented Kalman smoother (UKS) of \citet{sarkka2008urt}, using the continuous-discrete SDE prediction step of \citet{arasaratnam2010ckf}.
The measurement covariance of $\sigma_y^2$ was used for the UKS.

Optimal control techniques were used to solve the MAP and minimum energy estimation problems.
The estimation was transcribed to an nonlinear programming (NLP) problem using a third-order Legendre--Gauss--Lobatto direct collocation method equivalent to the Hermite--Simpson method \citep[Sec.~4.5]{betts2010pmo}.
The resulting NLP was then solved using the IPOPT solver of \cite{wachter2006iip}, part of the open-source COIN-OR project.
The discretized MAP state path estimation problems with the Euler and
trapezoidal schemes were also solved using the IPOPT solver.
The estimation time step of \np{5e-3} was used for all estimators.

The experiment and estimation were repeated \np{2000} times for each combination of $\sigma_o$, $p_o$ and the integrated squared error (ISE) was calculated for each estimate:
\begin{equation}
  \label{eq:ise_def}
  \textstyle
  \operatorname{ISE} =
  \int_0^T \paren{\big.X_t - x(t)}\trans\paren{\big.X_t - x(t)} \,\dd t,
\end{equation}
where $X$ is the simulated state and $x$ its estimate.
The results are summarized in Fig.~\ref{fig:robust_estimation}.
It can be seen that, in all scenarios, the performance of the minimum energy and MAP estimators was comparable.
However, the performance of the UKS becomes significantly worse as the outliers become more frequent and with a higher variance.
These results mirror those of \citet{aravkin2012rtf}.

It should be noted that although the ISE distribution of both the MAP and minimum energy estimates was similar, the estimates were often different.
This point can be seen in Fig.~\ref{fig:vdp_robust_timeseries}, which shows a portion of the simulated and estimated timeseries for one experiment.
It is not apparent, however, whether any one is intrinsically better statistically than the other, except in very specific scenarios.
\emph{They represent different statistics with different interpretations}; it is up to the user to determine which one is preferrable for each application.

\begin{figure}[tb]
  \centering
  \inputtikzpicture{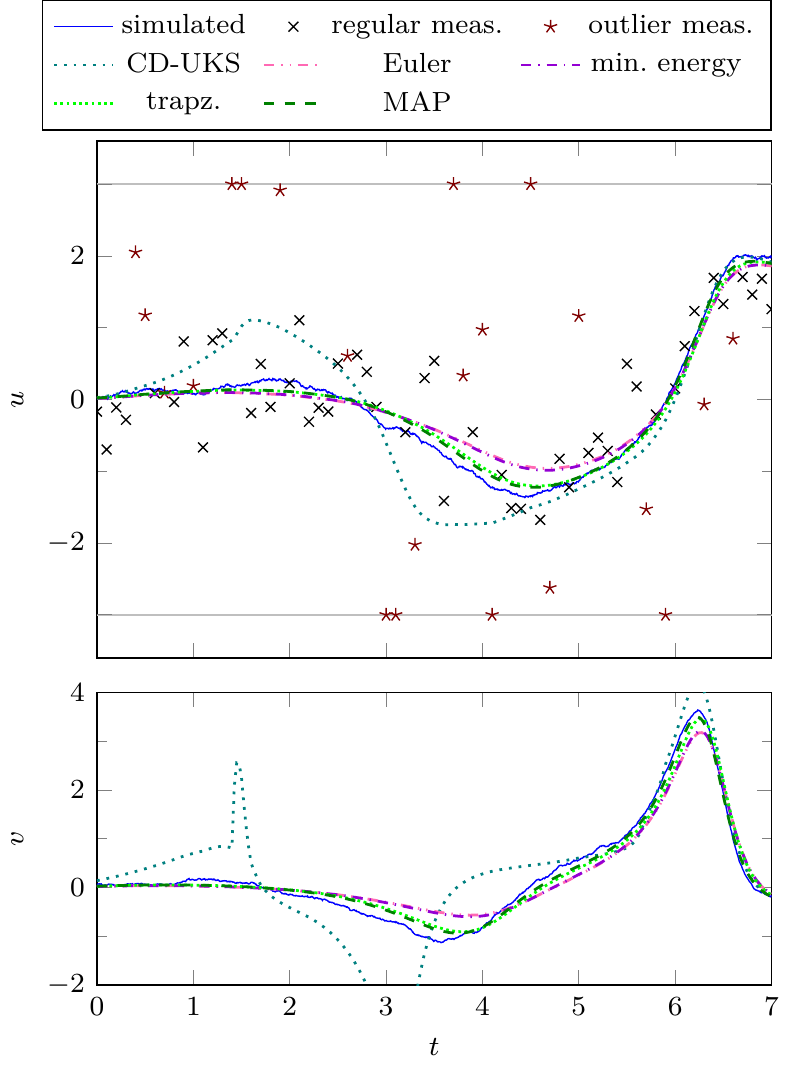}
  \caption{
    Example of a portion of the Van der Pol simulated experiment timeseries for $\sigma_o=3$, $p_o=0.25$.
    Outliers outside the plot range are represented at \np{+-3}.
    The MAP estimator favors paths with larger $u$.
  }
  \label{fig:vdp_robust_timeseries}
\end{figure}

\begin{sloppypar}
  We also note that, for systems with student's $t$-distribution as the
  measurement likelihood, the discretized MAP state path estimator with the
  Euler scheme is the $t$-robust smoother of \citet{aravkin2012rtf}.
  Therefore, a practical consequence of the hypo-convergence result of
  Thm.~\ref{th:main_theorem} is that, for sufficiently small discretization
  steps, the estimates of the $t$-robust smoother should be close to the
  minimum energy estimates.
  This was verified in the simulated experiments, and can be observed in
  Figs.~\ref{fig:robust_estimation} and \ref{fig:vdp_robust_timeseries}.
  The MAP state path estimates were also close to the discretized state path
  estimates with the trapezoidal discretization.
\end{sloppypar}

\section{Conclusions and Future Work}
\label{sec:conclusions}
In this paper, we addressed MAP state path estimation in continuous-discrete systems described by SDEs.
We began with a formal definition of mode for continuous-time stochastic processes so that the target statistic and its interpretation are well-defined.
We then showed how the Onsager--Machlup functional can be used to construct the MAP state path and the MAP noise path estimators, the latter of which is associated with the minimum energy state paths.
In this way, a clear statistical interpretation of the minimum energy state paths was provided, together with the conditions under which they coincide with the MAP state paths.

Additionally, we showed that the popular technique to solve these estimation problems by discretizing the dynamics and obtaining the discretized MAP state path estimate leads to different results depending on the discretization scheme used.
The Euler scheme, which is the simplest and therefore most widely used method for discretizing nonlinear SDEs, generates discretized estimators which hypo-converge to the MAP \emph{noise} path estimation.
The trapezoidal scheme, on the other hand, generates discretized estimators which hypo-converge to the MAP \emph{state} path estimation.

As mentioned in the end of Sec.~\ref{seq:discretized_map}, the difference between the energy and Onsager--Machlup functionals can be interpreted as accounting for the amplification of noise by the system.
The latter favours paths around more locally stable regions of the state-space, quantified by the sum of the eigenvalues of the drift Jacobian matrix $\nabla_x f$, i.e., the drift divergence $\diverg f$.
Paths around less stable regions of state space amplify more greatly the perturbations due to noise, which reduces the probability that the state stays inside small tubes around it.

Having proved the link between discretized state path estimation and variational problems also opens the possibility of using optimal control techniques to solve the former, as was done in the simulated examples.
These methods are better designed for the solution of infinite-dimensional optimization and are characterized by their order of convergence to the variational solution.
SDE discretization methods, on the other hand, are characterized by order of convergence in mean (or of the means) of the simulated and true stochastic processes \citep[Secs.~9.6 and 9.7]{kloeden1992nss}.
There also exist optimal control methods of arbitrarily high order; SDE methods for MAP calculation are limited to strong order \np{1.5} or weak order \np{4}.

The results herein can be generalized in a number of ways.
Rank-deficient $G$ matrices can be allowed by using the same approach as \citet{aihara1999mln}, for example.
Hypoconvergence can also be used to analyze the limits of discretized joint MAP state path and parameter estimation \citep{dutra2012jmaps}.
Finally, other popular higher order discretization schemes such as the order \np{1.5} It\=o--Taylor should generate problems which hypo-converge to MAP state path estimation as well.

\bibliographystyle{plainnat}
\bibliography{bibtex-compressed}

\end{document}